\documentclass[a4paper,UKenglish,cleveref, autoref, thm-restate]{lipics-v2021}

\pdfoutput=1 
\hideLIPIcs  

\title{Analysis of Regular Sequences: Summatory Functions
  and Divide-and-Conquer Recurrences}

\titlerunning{Analysis of Regular Sequences} 

\author{Clemens Heuberger\footnote{Corresponding author}}%
{Department of Mathematics, University of Klagenfurt, Austria \and
  \url{https://wwwu.aau.at/cheuberg}}%
{clemens.heuberger@aau.at}%
{https://orcid.org/0000-0003-0082-7334}%
{This research was funded in part by the Austrian Science Fund
  (FWF) [10.55776/DOC78]. For open access purposes, the author has applied a CC
  BY public copyright license to any author-accepted manuscript version arising
  from this submission.}

\author{Daniel Krenn}{Fachbereich Mathematik, Paris Lodron University of
  Salzburg, Austria \and \url{https://www.danielkrenn.at}}%
{math@danielkrenn.at}%
{https://orcid.org/0000-0001-8076-8535}%
{}

\author{Tobias Lechner}%
{Department of Mathematics, University of Klagenfurt, Austria}%
{toblechner@edu.aau.at}%
{}%
{}

\authorrunning{C. Heuberger, D. Krenn,
  and T. Lechner} 

\Copyright{Clemens Heuberger, Daniel Krenn,
  and Tobias Lechner} 



\category{} 

\relatedversion{} 

\supplementdetails[subcategory={Code for Example~\ref{example:find-minimum-maximum}}]{Software}{https://arxiv.org/src/2403.06589/anc} 

\nolinenumbers 

\EventEditors{C\'{e}cile Mailler and Sebastian Wild}
\EventNoEds{2}
\EventLongTitle{35th International Conference on Probabilistic, Combinatorial and Asymptotic Methods for the Analysis of Algorithms (AofA 2024)}
\EventShortTitle{AofA 2024}
\EventAcronym{AofA}
\EventYear{2024}
\EventDate{June 17--21, 2024}
\EventLocation{University of Bath, UK}
\EventLogo{}
\SeriesVolume{302}
\ArticleNo{14}

\usepackage{mathtools}
\usepackage{stmaryrd}
\usepackage{todonotes}

\DeclarePairedDelimiter{\abs}{\lvert}{\rvert}
\newcommand{\C}{\mathbb{C}}
\newcommand{\calG}{\mathcal{G}}
\DeclarePairedDelimiter{\ceiling}{\lceil}{\rceil}
\DeclarePairedDelimiter{\fractional}{\{}{\}}
\DeclarePairedDelimiter{\iverson}{\llbracket}{\rrbracket}
\DeclarePairedDelimiter{\floor}{\lfloor}{\rfloor}
\newcommand{\N}{\mathbb{N}}
\DeclarePairedDelimiter{\norm}{\lVert}{\rVert}
\DeclarePairedDelimiter{\set}{\{}{\}}
\DeclarePairedDelimiterX{\setm}[2]{\{}{\}}{#1\,\delimsize\vert\,\mathopen{}#2}
\newcommand{\Z}{\mathbb{Z}}

\begin{document}

\maketitle

\begin{abstract}
  In the asymptotic analysis of regular sequences as defined by Allouche and
  Shallit, it is usually advisable to study their summatory function because
  the original sequence has a too fluctuating behaviour. It might be that the
  process of taking the summatory function has to be repeated if the sequence
  is fluctuating too much. In this paper we show that for all regular
  sequences except for some degenerate cases, repeating this process finitely many times leads to a ``nice''
  asymptotic expansion containing periodic fluctuations whose Fourier
  coefficients can be computed using the results on the
  asymptotics of the summatory function of regular sequences by the first two
  authors of this paper.

  In a recent paper, Hwang, Janson, and Tsai perform a thorough investigation
  of divide-and-conquer recurrences. These can be seen as $2$-regular
  sequences. By considering them as the summatory function of their forward
  difference, the results on the
  asymptotics of the summatory function of regular sequences become
  applicable. We thoroughly investigate the case of a polynomial toll function.
\end{abstract}

\section{Introduction}
\subsection{Overview}\label{section:problem-statement}

The aim of~\cite{Heuberger-Krenn:2018:asy-regular-sequences} is the study of the asymptotic
behaviour of the summatory functions of regular
sequences~\cite{Allouche-Shallit:1992:regular-sequences}---in simplest terms, a
sequence $x$ is called $q$-regular for some integer $q\ge 2$
if there are square matrices $A_0$, \ldots, $A_{q-1}$, a row vector
$u$ and a column vector $w$ such that for all integers $n\ge 0$,
\begin{equation}\label{eq:introduction-matrix-product}
  x(n) = u A_{n_0}\ldots A_{n_{\ell-1}}w
\end{equation}
where $(n_{\ell-1}, \ldots, n_0)$ is the $q$ary expansion of $n$; an
alternative definition will be given in
Definition~\ref{definition:regular-sequence}. Regular sequences have been
introduced by Allouche and
Shallit~\cite{Allouche-Shallit:1992:regular-sequences}; a plethora of examples
have also been given in the same publication. We highlight two prototypical
examples at this point: the binary sum of digits function and the worst case
number of comparisons in merge sort.

The main result of~\cite{Heuberger-Krenn:2018:asy-regular-sequences} is that
the summatory function $N\mapsto \sum_{0\leq n < N} x(n)$ of a
$q$-regular sequence $x$ has an asymptotic expansion
\begin{equation}\label{eq:asymptotics-of-summatory-function-simple}
  \smashoperator{\sum_{0\leq n < N}} x(n) = \smashoperator{\sum_{\substack{\lambda\in\sigma(C)\\ \abs{\lambda} > R}}} N^{\log_{q}\lambda}
  \quad\smashoperator{\sum_{0\leq k < m_{C}(\lambda)}}\quad\frac{(\log N)^{k}}{k!}\ \Phi_{\lambda k}(\log_{q}N)
  + O\bigl(N^{\log_{q}R}(\log N)^\kappa\bigr)
\end{equation}
as $N\to\infty$, where the $\Phi_{\lambda k}$ are suitable $1$-periodic
continuous functions
and $\sigma(C)$, $m_C$, $R$, $\kappa$ are a set, a function, and two quantities,
respectively, depending on the regular sequence and which will be explained in
Theorem~\ref{thm:asymp} below. An algorithm is given to compute the Fourier
coefficients of the periodic functions. The main question is
whether there are $\lambda\in\sigma(C)$ with
$\abs{\lambda}>R$. In this case, we say that we established a \emph{good}
asymptotic expansion for the summatory function of
the regular
sequence. Otherwise,~\eqref{eq:asymptotics-of-summatory-function-simple}
reduces to an error term. Note that discussing the question of whether the
periodic fluctuations vanish is beyond the scope of this paper.

Studying the summatory function was motivated by the fact that in several
well-known examples of regular sequences, the sequences themselves are
fluctuating too much so that it is impossible to establish a good asymptotic expansion for
the regular sequence itself. For instance, for the binary sum of digits
function $s_2$, we have $s_2(2^k-1)=k$ and $s_2(2^k)=1$ for all integers $k\ge 0$, so the most precise asymptotic
expansion for $s_2(n)$ is $s_2(n)=O(\log n)$ for $n\to\infty$. However, the summatory
function might admit a good asymptotic expansion: For the summatory function of
the binary sum of digits function,
we have $\sum_{0\le n<N} s_2(n)=\frac12 N\log_2 N +
N\Phi(\log_2 N)$ for some 1-periodic continuous function $\Phi$ as $N\to\infty$; see
Delange~\cite{Delange:1975:chiffres}. In this particular example, there is no
error term; in general, an error term is to be expected.

However, \emph{a priori}, it is not clear whether the summatory function of a
regular sequence will be smooth enough so that a good asymptotic expansion can
be established. In fact, it is known~\cite[Theorems~2.6
and~2.5]{Allouche-Shallit:1992:regular-sequences} that the forward difference
$n\mapsto x(n+1)-x(n)$ of a regular sequence $x$ is again regular. This means
that the summatory function of the forward difference of the binary sum of
digits function equals the binary sum of digits function and no good asymptotic
expansion can be obtained. Thus,
as we are able to go forth (summatory function) and back (forward difference),
the question arises whether for
every regular sequence, there is a non-negative integer $k$ such that its
$k$-fold summatory function admits a good asymptotic expansion. In this
paper, we prove that this is the case for all regular sequences except for some
degenerate cases
(Theorem~\ref{theorem:k-fold-summatory-function}).

For other regular sequences, the sequence itself might admit a good
asymptotic expansion. One example are sequences associated with
divide-and-conquer schemes~\cite{Hwang-Janson-Tsai:2017:divide-conquer-half,
  Hwang-Janson-Tsai:ta:ident}, for example, the worst case analysis of the number of
comparisons in the merge sort algorithm. These are closely related to the
so-called ``master theorems''; see the discussion in~\cite{Hwang-Janson-Tsai:ta:ident}.
These sequences are easily
seen~\cite[Equation~(2.1)]{Hwang-Janson-Tsai:ta:ident} to be regular sequences
(as long as the toll function is regular). While Hwang, Janson, and Tsai~\cite{Hwang-Janson-Tsai:ta:ident}
provide a direct proof for the asymptotic behaviour and give plenty of
examples, the question is whether these results can also be obtained by using
the results in~\cite{Heuberger-Krenn:2018:asy-regular-sequences}. In the present
paper, we see such a sequence as the summatory function of its forward
difference, and we show that for polynomial toll functions, we get a good
asymptotic expansion in the vast majority of cases. The result is formulated in
Theorems~\ref{theorem:divide-conquer-polynomial-toll-function}
and~\ref{theorem:constant-toll-function}. In contrast
to~\cite{Hwang-Janson-Tsai:ta:ident}, we are not constrained to cases where the
toll function is asymptotically smaller than the sequence and Fourier
coefficients can be computed using the results
of~\cite{Heuberger-Krenn:2018:asy-regular-sequences}.

The remaining paper is structured as follows. In
Section~\ref{section:introduction-regular-sequences}, we recall the definition
and the relevant results on regular sequences. This is followed in
Section~\ref{section:introduction-summatory-functions} by the statement of our
new result on the $k$-fold summatory function. In
Section~\ref{section:introduction-divide-and-conquer-state-of-the-art}, we
present the state of the art for divide-and-conquer sequences and state our
version the result in
Section~\ref{section:introduction-divide-and-conquer-polynomial-toll-function}.
An explicit example is discussed in Section~\ref{section:example}.
Sections~\ref{section:proof-summatory-function}
and~\ref{section:proof-divide-conquer-polynomial-toll-function} are devoted to
the proofs of our theorems.

\subsection{Regular Sequences: Definition and State of the Art}\label{section:introduction-regular-sequences}
We recall the definition%
\footnote{Strictly speaking, this is an algorithmic characterisation of a
  regular sequence which is equivalent to the definition given by
  Allouche and Shallit~\cite{Allouche-Shallit:1992:regular-sequences},
  who first introduced this concept: they define a sequence~$x$ to be
  $q$-regular if the kernel
  \begin{equation*}
    \setm[\big]{x\circ (n\mapsto q^{j}n + r)}{
      \text{$j$, $r \in \N_0$ with $0\leq r < q^{j}$}}
  \end{equation*}
is contained in a finitely generated module.}
of a \emph{regular sequence}; see Allouche and
Shallit~\cite{Allouche-Shallit:1992:regular-sequences,
  Allouche-Shallit:2003:autom} for characterisations, properties,
and an abundance of examples.

\begin{definition}\label{definition:regular-sequence}
  Let $q\ge 2$ be an integer.
  A sequence $x\in \C^{\N_0}$ is said to be \emph{$q$-regular}%
  \footnote{In the standard
  literature, the basis is frequently denoted by
  $k$ instead of our $q$ here.}
  if there are a
  non-negative integer~$D$, a family $A=(A_r)_{0\le r<q}$
  of $D\times D$ matrices over~$\C$,
  a vector $u\in \C^{1\times D}$ and
  a vector-valued sequence $v\in (\C^{D\times 1})^{\N_0}$
  such that for all $n\in\N_0$, we have
  \begin{equation*}
    x(n)=uv(n),
  \end{equation*}
  and
  such that for all $0\le r<q$ and all $n\in\N_0$, we have
  \begin{equation}\label{equation:regular-sequence}
    v(qn+r)=A_rv(n).
  \end{equation}

  We call $(u, A, v(0))$ a \emph{linear representation} of~$x$ and~$v$
  the \emph{right vector-valued sequence} associated with this linear representation.
\end{definition}

Note that~\eqref{eq:introduction-matrix-product} easily follows
from~\eqref{equation:regular-sequence} by induction; the other direction is
contained in \cite[Lemma~4.1]{Allouche-Shallit:1992:regular-sequences}.

In~\cite{Heuberger-Krenn:2018:asy-regular-sequences} asymptotic
properties were studied. To formulate an abbreviated version of its main result, we first need to
recall the notion of the joint spectral radius of a set of square
matrices as bounds on matrix products are relevant in view of the
representation~\eqref{eq:introduction-matrix-product}. We fix a
vector norm~$\norm{\,\cdot\,}$ on $\C^{D}$ and consider its induced matrix
norm.

\begin{definition}Let $D$ be a positive integer and $\calG$ be a finite set of $D\times
  D$ matrices over $\C$.

  \begin{enumerate}
  \item The joint spectral radius of~$\calG$ is defined as
    \begin{equation*}
      \rho(\calG)\coloneqq \lim_{k\to\infty} \sup\setm{\norm{G_{1}\ldots
        G_{k}}^{1/k}}{G_1,\ldots, G_k\in\calG}.
    \end{equation*}
  \item We say that $\calG$ has the \emph{simple growth property} if
    \begin{equation*}
      \norm{G_1\ldots G_k}=O(\rho(\calG)^k)
    \end{equation*}
    holds for all $G_1$, \ldots, $G_k\in\calG$ and $k\to\infty$.
  \end{enumerate}

  For a family $G=(G_i)_{i\in I}$ of $D\times D$ matrices, we set
  $\rho(G)\coloneqq \rho(\setm{G_i}{i\in I})$ and we say that $G$ has the
  simple growth property if $\setm{G_i}{i\in I}$ has the simple growth property.
\end{definition}
We note that the joint spectral radius and the simple growth property are
independent of the chosen norm; cf.\ \cite[Remark~4.2]{Heuberger-Krenn-Lipnik:2021:asymp-analy-recur-sequen}.

For a square matrix~$M$, let $\sigma(M)$ denote the set of eigenvalues of~$M$
and by $m_{M}(\lambda)$ the size of the largest Jordan block of~$M$ associated
with some~$\lambda\in\C$. In particular, we have $m_{M}(\lambda) = 0$ if
$\lambda\notin\sigma(M)$. Finally, we let
$\fractional{z} \coloneqq z - \floor{z}$ denote the fractional part of a real
number~$z$. We use Iverson's convention: For a statement~$S$, we set
$\iverson{S} = 1$ if~$S$ is true and~$0$ otherwise; see also Graham, Knuth and
Patashnik~\cite[p.~24]{Graham-Knuth-Patashnik:1994}.

\begin{theorem}[{\cite[Theorem~A]{Heuberger-Krenn:2018:asy-regular-sequences}},
  {\cite{Dumas:2013:joint, Dumas:2014:asymp}}]
  \label{thm:asymp}
  Let~$x$ be a $q$-regular sequence with linear representation
  $(u, A, w)$, and set
  \begin{align*}
    B_r&\coloneqq \sum_{0\le s<r}A_s,&
    C&\coloneqq \smashoperator{\sum_{0\leq s< q}}A_{s}
  \end{align*}
  for $0\le r<q$.

  We choose $R > 0$ as follows: If
  $A$ has the simple growth property,
  then we set $R=\rho(A)$. Otherwise, we choose
  $R > \rho(A)$ such that there is no
  eigenvalue~$\lambda\in\sigma(C)$ with $\rho(A) < \abs{\lambda}\leq
  R$.

  Then we have
  \begin{multline}\label{eq:main-asymptotic-expansion}
    \smashoperator{\sum_{0\leq n < N}} x(n) = \smashoperator{\sum_{\substack{\lambda\in\sigma(C)\\ \abs{\lambda} > R}}} N^{\log_{q}\lambda}
    \quad\smashoperator{\sum_{0\leq k < m_{C}(\lambda)}}\quad\frac{(\log N)^{k}}{k!}\ \Phi_{\lambda k}(\fractional{\log_{q}N})\\
    + O\bigl(N^{\log_{q}R}(\log N)^{\max\set{m_{C}(\lambda)\colon \abs{\lambda} = R}}\bigr)
  \end{multline}
  as $N\to\infty$, where $\Phi_{\lambda k}$ are suitable $1$-periodic functions. If
  there are no eigenvalues $\lambda\in\sigma(C)$ with
  $\abs{\lambda} \le R$, the $O$-term can be omitted.

  For $\abs{\lambda} > R$ and $0\leq k < m_{C}(\lambda)$, the
  function~$\Phi_{\lambda k}$ is Hölder continuous with any exponent
  smaller than~$\log_{q}(\abs{\lambda}/R)$.
\end{theorem}
Note that~\cite{Heuberger-Krenn:2018:asy-regular-sequences} also contains
results on how to compute the Fourier coefficients of the periodic fluctuations
$\Phi_{\lambda k}$.

\subsection{Summatory Functions of Regular Sequences}\label{section:introduction-summatory-functions}

As announced in Section~\ref{section:problem-statement}, within this paper, we
show that for all regular sequences except for some degenerate cases, there is a non-negative integer $k$
such that the $k$-fold summatory function admits a good asymptotic
expansion. In order to formulate our result, we first fix a notation for
summatory functions.

\begin{definition}
  For a sequence $x\colon \N_0\to\C^D$ (for some positive integer $D$), define
  the sequence $\Sigma x\colon \N_0\to\C^D$ by
  \begin{equation*}
    (\Sigma x)(N)=\sum_{0\le n<N}x(n).
  \end{equation*}
  We use the convention that $\Sigma$ binds more strongly than evaluation,
  i.e., we write $\Sigma x (N)$ instead of $(\Sigma x)(N)$.
\end{definition}

We are now able to formulate our result.
\begin{theorem}\label{theorem:k-fold-summatory-function}
  Let $x$ be a $q$-regular sequence with linear representation $(u, A,
  w)$ and set $C\coloneqq \sum_{0\le r<q}A_r$. Assume that $C$ has a non-zero eigenvalue.

  Then there is a non-negative integer $k$ such that $\Sigma^k x$ admits a good
  asymptotic expansion.
\end{theorem}
This theorem is proved in Section~\ref{section:proof-summatory-function}.

\subsection{Divide-and-Conquer Sequences: Definition and State of the Art}\label{section:introduction-divide-and-conquer-state-of-the-art}
Hwang, Janson, and Tsai~\cite{Hwang-Janson-Tsai:ta:ident} study sequences $x$
with
\begin{equation}\label{eq:divide-and-conquer}
  x(n) = \alpha x\Bigl(\floor[\Big]{\frac n2}\Bigr) + \beta x\Bigl(\ceiling[\Big]{\frac
    n2}\Bigr) + g(n)
\end{equation}
for $n\ge 2$, where $\alpha$ and $\beta$ are two given positive constants,
$g$ is a given function, called the toll function, and $x(1)$ is given.

The simplest version of their result is summarised in the following theorem;
more general (weaker assumptions on $g$) versions are also available.

\begin{theorem}[{\cite[Corollary~2.14]{Hwang-Janson-Tsai:ta:ident}}]Let $x$ be
  a sequence satisfying~\eqref{eq:divide-and-conquer}.
  Assume that there is an $\varepsilon>0$ such that
  $g(n)=O(n^{\log_2(\alpha+\beta)-\varepsilon})$. Then
  \begin{equation*}
    x(n) = n^{\log_2(\alpha+\beta)}\Phi(\fractional{\log_2 n}) + O(n^{\log_2(\alpha+\beta)-\varepsilon})
  \end{equation*}
  for $n\to\infty$ where $\Phi$ is a continuous, $1$-periodic function.
\end{theorem}

\subsection{Divide-and-Conquer Sequences: Polynomial Toll Function}\label{section:introduction-divide-and-conquer-polynomial-toll-function}

For divide-and-conquer sequences $g$, a sequence satisfying the
recurrence~\eqref{eq:divide-and-conquer} can be seen as a regular sequence:
It is not hard to see that we have
\begin{equation}
  \label{eq:divide-and-conquer-as-recursive}
  \begin{aligned}
    x(2n)&=(\alpha+\beta)x(n)+g(2n),\\
    x(2n+1)&=\alpha x(n)+\beta x(n+1)+g(2n+1)
  \end{aligned}
\end{equation}
for $n\ge 1$. Thus $x$ is a $2$-recursive sequence in the sense of
\cite{Heuberger-Krenn-Lipnik:2021:asymp-analy-recur-sequen} and therefore
$2$-regular by
\cite[Corollary~D]{Heuberger-Krenn-Lipnik:2021:asymp-analy-recur-sequen}.
Alternatively, a linear representation for $x$ can also be constructed directly
from~\eqref{eq:divide-and-conquer-as-recursive}: the associated right
vector-valued sequence consists of $n\mapsto x(n)$, $n\mapsto x(n+1)$ and the right
vector-valued sequence associated to a linear representation of $g$; the
matrices of the linear representation can then easily be reconstructed
from~\eqref{eq:divide-and-conquer-as-recursive} and the linear representation
of $g$. The fact that~\eqref{eq:divide-and-conquer-as-recursive} holds only for
$n\ge 1$ (instead of $n\ge0$) can be fixed; see \cite[Proof of
Lemma~4.1]{Allouche-Shallit:1992:regular-sequences}
or~\cite[Theorem~B]{Heuberger-Krenn-Lipnik:2021:asymp-analy-recur-sequen}.

As announced in Section~\ref{section:problem-statement}, our goal is to see
what can be said about the asymptotics of $x(n)$ for $n\to\infty$ using
Theorem~\ref{thm:asymp}. While the method works for arbitrary regular toll
functions; see Remark~\ref{remark:divide-and-conquer-general-remark} (although
good asymptotic expansions cannot be guaranteed in all cases), we formulate our
main result for polynomial toll functions; first versions are contained in the
master's thesis~\cite{Lechner:2024:applic-theor} of the third author.

\begin{theorem}\label{theorem:divide-conquer-polynomial-toll-function}
  Let $g(n)= \sum_{i=0}^{k}c_{i}n^i$ be a polynomial of degree $k\ge 1$, $x$
  be a sequence satisfying~\eqref{eq:divide-and-conquer}.
  Then the asymptotic behaviour of $x(n)$ for $n\to\infty$ can be described as
  follows, where $\Phi$ and $\Psi$ are
  $1$-periodic continuous functions.
  \begin{itemize}
  \item \textbf{Case 1a.}
    If $\alpha+\beta > 2^k$ and $2^k > \max \set{\alpha, \beta}$, then
    \begin{equation*}
      x(n) = n^{\log_{2}(\alpha+\beta)} \Phi(\fractional{\log_{2}n}) + n^{k} \Psi(\fractional{\log_{2}n})+ O(n^{\log_{2}\max \set{\alpha, \beta}}).
    \end{equation*}
  \item \textbf{Case 1b.}
    If $\alpha+\beta > 2^k$ and $\max \set{\alpha, \beta} \ge 2^k$, then
    \begin{equation*}
      x(n) = n^{\log_{2}(\alpha+\beta)} \Phi
               (\fractional{\log_{2}n}) +  O(n^{\log_{2}\max \set{\alpha,
               \beta}}(\log n)^{\iverson{\max\set{\alpha, \beta}=2^k}}).
    \end{equation*}
  \item \textbf{Case 2.}
    If $\alpha+\beta = 2^k$, then
    \begin{equation*}
      x(n) = n^{k} (\log n) \Phi(\fractional{\log_{2}n})+n^{k} \Psi(\fractional{\log_{2}n})
             + O(n^{\log_{2} \max \set{\alpha, \beta}+\iverson{\alpha=\beta}\varepsilon})
    \end{equation*}
    for any $\varepsilon>0$.
  \item \textbf{Case 3.}
    If $2^k > \alpha+\beta > 2^{k-1}$, then
    \begin{multline*}
      x(n) = n^{k} \Phi(\fractional{\log_{2}n}) + n^{\log_{2}(\alpha+\beta)}
      \Psi(\fractional{\log_{2}n})\\+ O(n^{\log_{2}\max
        \set{\alpha,\beta,2^{k-1}}+\iverson{\max
          \set{\alpha,\beta}=2^{k-1}}\varepsilon}(\log n)^{\iverson{\max\set{\alpha, \beta}<2^{k-1}}})
    \end{multline*}
    for any $\varepsilon>0$.
  \item \textbf{Case 4.}
    If $2^{k-1} \ge \alpha+\beta $, then
    \begin{equation*}
      x(n)= n^{k} \Phi(\fractional{\log_{2}n}) + O(n^{k-1}(\log n)^E),
    \end{equation*}
    where
    \begin{equation*}
      E \coloneqq
      1+\iverson{\alpha+\beta=2^{k-1}}
      (\iverson{k\ge 2 \text{ and } c_{k-1}\neq 0}
      + \iverson{k=1 \text{ and } d_{0}+d_1\neq 0})
    \end{equation*}
    with
    \begin{align*}
      d_0&\coloneqq (1-\beta)x(1)-g(1)+g(0),&
      d_1&\coloneqq g(1)-(1-\beta)x(1).
    \end{align*}
  \end{itemize}
\end{theorem}
This theorem is proved in Section~\ref{section:proof-divide-conquer-polynomial-toll-function}.

The case of a constant toll function is somewhat simpler.

\begin{theorem}\label{theorem:constant-toll-function}
  Let $g(n)= c_0$ be a constant toll function and let $x$
  be a sequence satisfying~\eqref{eq:divide-and-conquer}. Let $d_0$ and $d_1$
  be defined as in
  Theorem~\ref{theorem:divide-conquer-polynomial-toll-function}.
  Then the asymptotic behaviour of $x(n)$ for $n\to\infty$ can be described as
  follows, where $\Phi$ is a
  $1$-periodic continuous function.
  \begin{itemize}
  \item \textbf{Case 1.}
    If $d_0=d_1=0$, then
    \begin{equation*}
      x(n) = n^{\log_{2}(\alpha+\beta)} \Phi(\fractional{\log_{2}n}).
    \end{equation*}
  \item \textbf{Case 2a.}
      If $d_0\neq 0$ or $d_1\neq 0$, and $\alpha+\beta>1$, then
      \begin{multline*}
        x(n) = n^{\log_{2}(\alpha+\beta)} \Phi(\fractional{\log_{2}n}) \\+
        O(n^{\log_{2}\max \set{\alpha, \beta, 1}+\iverson{\max\set{\alpha, \beta}=1}\varepsilon}
        (\log n)^{\iverson{\max\set{\alpha,\beta}< 1}})
      \end{multline*}
       for any $\varepsilon>0$.
  \item \textbf{Case 2b.}
      If $d_0\neq 0$ or $d_1\neq 0$, and $\alpha+\beta\le 1$, then
      \begin{equation*}
        x(n) = O((\log n)^{\iverson{\alpha+\beta=1 \text{ and } d_0+d_1\neq 0}}).
      \end{equation*}
  \end{itemize}
\end{theorem}
This theorem is also proved in Section~\ref{section:proof-divide-conquer-polynomial-toll-function}.

\subsection{Example}\label{section:example}
\begin{figure}[tbp]
  \includegraphics[width=\textwidth]{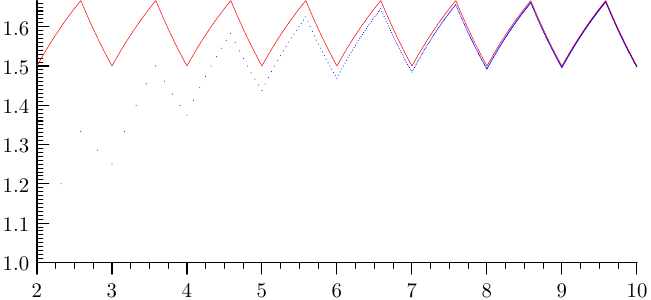}
  \caption{Comparison of the $1$-periodic function $\Phi$ in the
    asymptotic expansion determined using
    Theorem~\ref{theorem:constant-toll-function} with the empirical values of
    the sequence.}
  \label{fig:minmax_example}
\end{figure}
We conclude this introductory section with one example to illustrate the results.
\begin{example}\label{example:find-minimum-maximum}
  Consider the divide-and-conquer algorithm for finding the minimum and the
  maximum of a list of $n$ elements. The number $x(n)$ of comparisons needed
  satisfies~\eqref{eq:divide-and-conquer} for $n\ge 3$ with $a=b=1$ and $g(n)=2$
  for $n\ge 3$ and with $x(1)=0$ and $x(2)=1$; cf.\
  \cite[Example~3.2]{Hwang-Janson-Tsai:2017:divide-conquer-half}.

  By Theorem~\ref{theorem:constant-toll-function} (and
  \cite[Theorem~B]{Heuberger-Krenn-Lipnik:2021:asymp-analy-recur-sequen} to deal with the
  fact that the divide-and-conquer recurrence is only valid for $n\ge 3$
  instead of $n\ge 2$; see Section~\ref{section:details-example} for details), we get
  \begin{equation*}
    x(n)=n\Phi(\fractional{\log_2 n}) + O(n^{\varepsilon})
  \end{equation*}
  for some $1$-periodic continuous function $\Phi$ and any $\varepsilon>0$. The Fourier coefficients of $\Phi$ can be computed;
  cf.\ Figure~\ref{fig:minmax_example}.
\end{example}

\section{Summatory Functions: Proof of Theorem~\ref{theorem:k-fold-summatory-function}}\label{section:proof-summatory-function}
Before proving Theorem~\ref{theorem:k-fold-summatory-function}, we collect two
lemmata on the linear representations of summatory functions and $k$-fold
summatory functions.

The following lemma is implicitly shown in~\cite[Lemma~12.2]{Heuberger-Krenn:2018:asy-regular-sequences}, however, it is crucial for our
purposes, so we provide a precise formulation and
will prove it for self-containedness.

\begin{lemma}\label{lemma:summation}
  Let $x$ be a $q$-regular sequence with linear representation $(u, A,
  w)$ and associated right vector-valued sequence~$v$. Set
  \begin{equation}\label{eq:definitions-B-C}
    B_r\coloneqq \sum_{0\le s<r}A_s\text{ for $0\le r< q$ and }C\coloneqq \sum_{0\le s<q}A_s.
  \end{equation}

  Then we have
  \begin{equation}\label{eq:sigma-v-recurrence-simple}
    \Sigma v(qN+r)=C\,\Sigma v(N) + B_r v(N)
  \end{equation}
  for all $N\ge 0$ and $0\le r<q$.

  Additionally, $\Sigma x$ is regular with linear representation $(\widetilde u,
  \widetilde A, \widetilde w)$ with
  \begin{align*}
    \widetilde u&\coloneqq (u, 0),\\
    \widetilde A_r&\coloneqq \begin{pmatrix}C&B_r\\0&A_r\end{pmatrix}\text{ for $0\le r<q$},\\
    \widetilde w&\coloneqq \begin{pmatrix}0\\w\end{pmatrix};
  \end{align*}
  the associated right vector-valued sequence is
  $\bigl(\begin{smallmatrix}\Sigma v\\v\end{smallmatrix}\bigr)$.
\end{lemma}

\begin{proof}
  By definition, we have $\Sigma v(0)=0$. Let $N\ge 0$
  and $0\le r<q$. Then
  \begin{equation*}
    \Sigma v(qN+r)=\sum_{0\le n<qN} v(n) + \sum_{qN \le n< qN+r}v(n). 
  \end{equation*}
  Replacing $n$ by $qm+s$ for $m\in\Z$ and $0\le s<q$ in the first sum and
  replacing $n$ by $qN+s$ for $0\le s<r$ in the second sum yields
  \begin{equation*}
    \Sigma v(qN+r)=\sum_{0\le m<N}\sum_{0\le s<q} v(qm+s) + \sum_{0 \le s< r}v(qN+s). 
  \end{equation*}
  Using the linear representation yields
  \begin{equation*}
    \Sigma v(qN+r)=\sum_{0\le m<N}\sum_{0\le s<q} A_s v(m) + \sum_{0 \le s<
      r}A_sv(N)=C\,\Sigma v(N) + B_r v(N).
  \end{equation*}
  In other words, we have shown~\eqref{eq:sigma-v-recurrence-simple}.

  As we have $x=uv$, we also have
  \[\Sigma x=u\,\Sigma v=u\,\Sigma v+0=\widetilde u \begin{pmatrix}\Sigma v\\v\end{pmatrix}.
  \]
  We conclude that $\Sigma x$ has the given linear representation and
  associated right vector-valued sequence.
\end{proof}

\begin{remark}In \cite[Lemma 12.2]{Heuberger-Krenn:2018:asy-regular-sequences},
  a very similar result appears in Equation~(12.1) there. The difference
  between Lemma~\ref{lemma:summation} and that equation is an additional
  summand $(I-A_0)\iverson{qN+r>0}$.

  The reason for the additional summand is that in general, $f(0)=A_0 f(0)$
  (with the notations there) does not hold, cf.\
  also~\cite{Heuberger-Krenn-Lipnik:2022:nodoi-lipics} for a discussion
  of this condition.
\end{remark}

Iterating the results in Lemma~\ref{lemma:summation} leads to the following lemma.

\begin{lemma}\label{lemma:iterated-summation}
  Let $x$ be a $q$-regular sequence with linear representation $(u, A,
  w)$, $k\ge 1$, and use the notations
  from~\eqref{eq:definitions-B-C}.
  Then $\Sigma^kx$ is $q$-regular with linear representation $(\widetilde u,
  \widetilde A, \widetilde w)$ where $\widetilde A_r$ is a block
  upper triangular matrix with diagonal blocks $q^{k-1}C$, $q^{k-2}C$, \dots,
  $qC$, $C$, $A_r$ for $0\le r<q$ and $\widetilde u$ and $\widetilde w$
  are vectors; the associated right vector-valued
  sequence is $(\Sigma^kv, \Sigma^{k-1}v, \ldots, \Sigma v, v)^\top$.
\end{lemma}
\begin{proof}
  We claim that for $m\ge 1$, there
  are matrices $M_{m,0}$, \ldots, $M_{m,m-1}$ such that
  \begin{equation}\label{eq:sigma-v-recurrence-iterated}
    \Sigma^mv(qN+r) = q^{m-1}C\, \Sigma^mv(N) + \sum_{0\le j<m}M_{m,j}\, \Sigma^j v(N)
  \end{equation}
  holds for all $N\ge 0$ and $0\le r<q$.

  We show~\eqref{eq:sigma-v-recurrence-iterated} by induction on $m$. For
  $m=1$, this is~\eqref{eq:sigma-v-recurrence-simple}. To
  show~\eqref{eq:sigma-v-recurrence-iterated} for $m$ replaced by $m+1$, we
  use~\eqref{eq:sigma-v-recurrence-simple} for $v$ replaced by the regular sequence
  with associated right vector-valued sequence
  $\bigl(\begin{smallmatrix}\Sigma v\\v \end{smallmatrix}\bigr)$ studied in
  Lemma~\ref{lemma:summation} and the linear representation given there.
  We obtain
  \begin{equation*}
    \Sigma^m\begin{pmatrix}\Sigma v\\v \end{pmatrix}(qN+r)=q^{m-1}
    \begin{pmatrix}
      qC&\sum_{0\le r<q}B_r\\
      0& C
    \end{pmatrix}\, \Sigma^m\begin{pmatrix}\Sigma v\\v \end{pmatrix}(N)
    + \sum_{0\le j<m}\widetilde{M}_{m,j}\, \Sigma^j \begin{pmatrix}\Sigma v\\v \end{pmatrix}(N)
  \end{equation*}
  for suitable matrices $\widetilde{M}_{m,j}$ for $0\le j<m$. Considering the
  first block row of this equation and collecting terms by powers of $\Sigma$
  leads to~\eqref{eq:sigma-v-recurrence-iterated} with $m$ replaced by $m+1$.

  Using~\eqref{eq:sigma-v-recurrence-iterated} for $1\le m\le k$ yields the
  linear representation as described in the lemma.
\end{proof}

\begin{proof}[Proof of Theorem~\ref{theorem:k-fold-summatory-function}]
  Let $\rho$ be the joint spectral radius of $A$ and $r$ the
  spectral radius (largest absolute value of an eigenvalue) of $C$.

  For some fixed
  $k$ which will be chosen appropriately later,
  Lemma~\ref{lemma:iterated-summation} yields a linear representation of
  $\Sigma^k x$ with the properties given there. Let $\widetilde C\coloneqq
  \sum_{0\le r<q}\widetilde A_r$.

  The $k$-fold summatory function $\Sigma^k x$ admits a good asymptotic
  expansion
  if the spectral radius of $\widetilde C$ is larger than the
  joint spectral radius of $\widetilde A$. So we compute both.

  By Lemma~\ref{lemma:iterated-summation}, $\widetilde C$ is a block upper triangular matrix with
  diagonal blocks $q^k C$, \ldots, $qC$, $C$.
  So the spectral radius of
  $\widetilde C$ is $q^k r$.

  The joint spectral radius of a family of block upper
  triangular matrices is the maximum of the joint spectral radii of the
  diagonal blocks;
  see~\cite[Proposition~1.5]{Jungers:2009:joint-spectral-radius}. This implies
  that the joint spectral radius of $\widetilde A$ is
  $\max\set{q^{k-1}r, q^{k-2}r, \ldots, qr, r, \rho}=\max\set{q^{k-1}r, \rho}$.

  It is clear that $q^kr>q^{k-1}r$ so the only condition which needs to be
  satisfied is $q^kr > \rho$. Such a $k$ exists because $r>0$ (as $C$ has a
  non-zero eigenvalue).
\end{proof}

\section{Divide-and-Conquer Recurrences: Proof of
  Theorems~\ref{theorem:divide-conquer-polynomial-toll-function}
  and~\ref{theorem:constant-toll-function}}\label{section:proof-divide-conquer-polynomial-toll-function}

For the proof of Theorems~\ref{theorem:divide-conquer-polynomial-toll-function}
and \ref{theorem:constant-toll-function},
we will first consider a general regular toll function $g$ and summarise our
findings in the general case in
Remark~\ref{remark:divide-and-conquer-general-remark}.
Afterwards, we will specialise
to a polynomial toll function.

As announced in Section~\ref{section:problem-statement}, we write $x$ as the
summatory function of the forward difference of $x$, i.e.,
\begin{equation*}
  x(N)=\sum_{0\le n<N}(x(n+1)-x(n)) + x(0)
\end{equation*}
for $N\ge 0$.
Note that strictly speaking, $x(0)$ is not defined
in~\eqref{eq:divide-and-conquer}. However, we may assume that $g(1)$ and $g(0)$
are somehow defined: they are not used in~\eqref{eq:divide-and-conquer}, but
we can extend the definition of $g$ if $g(0)$ and $g(1)$ should be undefined.

In order to use Theorem~\ref{thm:asymp}, we need a linear representation of the
forward difference of $x$. A first step is the following lemma.
\begin{lemma}\label{lemma:forward-difference-generic}
  Let $x$ be a sequence satisfying~\eqref{eq:divide-and-conquer} for some toll
  function $g$ and set
  $x(0)\coloneqq 0$ and $h(n)\coloneqq x(n+1)-x(n)$ for $n\ge 0$. Then
  \begin{equation}\label{eq:h-recurrence}
    \begin{aligned}
      h(2n)&=\beta h(n) + g(2n+1)-g(2n)+d_0\delta_0(n),\\
      h(2n+1)&=\alpha h(n)+g(2n+2)-g(2n+1)+d_1\delta_0(n)
    \end{aligned}
  \end{equation}
  for $n\ge 0$ with $d_0$, $d_1$ as in Theorem~\ref{theorem:divide-conquer-polynomial-toll-function}
  and $\delta_0(n)\coloneqq \iverson{n=0}$ for $n\ge 0$.
\end{lemma}
\begin{proof}
We can
rewrite~\eqref{eq:divide-and-conquer-as-recursive} as
\begin{equation}
  \label{eq:divide-and-conquer-as-recursive-no-offset}
  \begin{aligned}
    x(2n)&=(\alpha+\beta)x(n)+g(2n)-g(0)\delta_0(n),\\
    x(2n+1)&=\alpha x(n)+\beta x(n+1)+g(2n+1) +((1-\beta)x(1)-g(1))\delta_0(n)
  \end{aligned}
\end{equation}
for $n\ge 0$.

Then~\eqref{eq:h-recurrence} follows from
\begin{align*}
      h(2n)&=x(2n+1)-x(2n),\\
      h(2n+1)&=x(2n+2)-x(2n+1)
\end{align*}
and inserting~\eqref{eq:divide-and-conquer-as-recursive-no-offset} into these equations.
\end{proof}

\begin{remark}\label{remark:divide-and-conquer-general-remark}
  Note that $h$ occurs only as $h(n)$ on the right-hand side of
  of~\eqref{eq:h-recurrence}. Therefore, a right vector-valued sequence associated
  with $h$ can be constructed by using $h(n)$ in its first component, then
  whatever is needed to express $g(2n+2)-g(2n+1)$ and $g(2n+1)-g(2n)$, followed
  by $\delta_0(n)$. The matrices $A_0$ and $A_1$ of the linear representation will thus be in
  block triangular form; the upper left diagonal elements of $A_0$ and $A_1$ being $\beta$ and
  $\alpha$, respectively.

  Assume that the contributions of $g$ to the linear representation are small
  in comparison to $\alpha$ and $\beta$.
  Then   the joint spectral radius of the linear representation of $h$ will be
  $\max\set{\alpha, \beta}$, whereas the matrix $C$ as in Theorem~\ref{thm:asymp}
  will have a dominant eigenvalue $\alpha+\beta$, which is larger than the
  joint spectral radius. Thus we will have a good asymptotic expansion in this case.
\end{remark}

We now turn to the case of a polynomial toll function.

\begin{lemma}\label{lemma:h-linear-representation}
  Let $x$ and $h$ be as in Lemma~\ref{lemma:forward-difference-generic} with a
  polynomial toll function $g(n)=\sum_{i=0}^k c_in^i$ for some $k\ge 0$ and
  some constants $c_0$, \ldots, $c_k$ with $c_k\neq 0$.

  Set
  \begin{align*}
    b_{0j}&\coloneqq \sum_{i=j+1}^k\binom{i}{j}2^jc_i,&
    b_{1j}&\coloneqq \sum_{i=j+1}^k\binom{i}{j}(2^i-2^j)c_i,\\
    a_{0ij}&\coloneqq \iverson{j=i}2^i,&
    a_{1ij}&\coloneqq \binom{i}{j}2^j
  \end{align*}
  for $0\le i<k$ and $0\le j<k$ and
  \begin{align*}
    b_r&\coloneqq (b_{r(k-1)}, \ldots, b_{r0}),&
    \widetilde{A}_r&\coloneqq (a_{rij})_{\substack{i=k-1,\ldots, 0\\j=k-1,\ldots,
    0}},\\
    \mu_r&\coloneqq \begin{cases}\beta&\text{ if }r=0,\\\alpha&\text{ if }r=1,\end{cases}&
    A_r&\coloneqq \begin{pmatrix}
      \mu_r&b_r&d_r\\
      0&\widetilde{A}_r&0\\
      0&0&\iverson{r=0}
    \end{pmatrix}
  \end{align*}
  for $r\in\set{0, 1}$ and
  \begin{align*}
    u&\coloneqq (1, 0, \ldots, 0)\in\C^{1\times (k+2)},&
    w&\coloneqq
       \begin{cases}
         (x(1), 0, \ldots, 0, 1, 1)^\top\in\C^{(k+2)\times 1}&\text{ if }k\ge 1,\\
         (x(1), 1)^\top\in\C^{2\times 1}&\text{ if }k=0.
       \end{cases}
  \end{align*}
  Then $(u, A, w)$ is a linear representation for $h$.

  If $d_0=d_1=0$, then $(\widetilde u, \widetilde A, \widetilde w)$ is also a
  linear representation for $h$ where $\widetilde u$, $\widetilde A$, and
  $\widetilde w$ arise from from $u$, $A$, and $w$ by removing the last column,
  the last row and column, and the last row, respectively.
\end{lemma}
We remark that $A_r$ is an upper triangular matrix for $r\in\set{0, 1}$ because
$\binom ij=0$ for $j>i$ (and indices in $\widetilde A_0$ and $\widetilde A_1$ are decreasing).
\begin{proof}
  As a
  right vector-valued sequence $v$, we choose
  \begin{equation*}
    n \mapsto (h(n), n^{k-1}, \ldots, 1, \delta_0(n))^{\top}.
  \end{equation*}
  We immediately check that $v(0)=w$ and that $uv(n)=h(n)$ holds for all $n\ge 0$.

  Using the binomial theorem repeatedly, we get
  \begin{alignat*}{3}
    g(2n+1)-g(2n)&=\sum_{i=0}^kc_i\sum_{j=0}^{i-1}\binom{i}{j}2^jn^j&&=\sum_{j=0}^{k-1}b_{0j}n^j,\\
    g(2n+2)-g(2n+1)&=\sum_{i=0}^kc_i\sum_{j=0}^{i-1}\binom{i}{j}2^jn^j(2^{i-j}-1)&&=\sum_{j=0}^{k-1}b_{1j}n^j,\\
    (2n)^i&=2^in^i&&=\sum_{j=0}^{k-1}a_{0ij}n^j &\qquad &\text{for } 0\le i< k,\\
    (2n+1)^i&=\sum_{j=0}^{i}\binom{i}{j}2^{j}n^j&&=\sum_{j=0}^{k-1}a_{1ij}n^j &\qquad &\text{for } 0\le i< k.
  \end{alignat*}

  We verify that~\eqref{eq:h-recurrence} translates into
  \begin{equation*}
    v(2n+r)=A_rv(n)
  \end{equation*}
  for $r\in\set{0, 1}$ and $n\ge 0$.

  If $d_0=d_1=0$, then $A_0$ and $A_1$ are block diagonal matrices. The lower
  right block is not taken into account when multiplying by $u$, so the lower
  right block can be omitted.

  Thus the result follows.
\end{proof}

\begin{lemma}\label{lemma:h-spectrum-joint-spectral-radius}
  Let $x$ and $h$ be as in Lemma~\ref{lemma:forward-difference-generic} and
  $g$, $(u, A, w)$ be as in Lemma~\ref{lemma:h-linear-representation}. Assume
  that $k\ge 1$. Set $C=A_0+A_1$.

  Then $\rho(A)=\max\set{\alpha, \beta, 2^{k-1}}$ and
  $\sigma(C)=\set{\alpha+\beta, 2^k, 2^{k-1}, \ldots, 2, 1}$. If $\max\set{\alpha,
  \beta}\neq 2^{k-1}$, then $A$ has the simple growth property.
\end{lemma}
\begin{proof}
  By Lemma~\ref{lemma:h-linear-representation}, $A_0$ is an upper triangular
  matrix with diagonal elements $\beta$, $2^{k-1}$, \dots, $1$, $1$ and $A_1$
  is an upper triangular matrix with diagonal elements $\alpha$, $2^{k-1}$,
  \dots, $1$, $0$.

  The joint spectral radius of a set of upper triangular matrices is the
  maximum of the diagonal elements of the matrices; see
  \cite[Proposition~1.5]{Jungers:2009:joint-spectral-radius}. This implies that
  $\rho(A)=\max\set{\alpha, \beta, 2^{k-1}}$.
  By~\cite[Lemma~4.5]{Heuberger-Krenn-Lipnik:2021:asymp-analy-recur-sequen},
  $A$ has the simple growth property if the joint spectral radius of $A$ occurs
  only once as a maximum of corresponding diagonal elements of $A_0$ and
  $A_1$. This is the case if $\max\set{\alpha, \beta}\neq 2^{k-1}$.

  We also conclude that $C$ is an upper triangular matrix with diagonal elements
  $\alpha+\beta$, $2^k$, $2^{k-1}$, \dots, $2$, $1$. Thus the assertion for
  $\sigma(C)$ follows.
\end{proof}

\begin{proof}[Proof of
  Theorem~\ref{theorem:divide-conquer-polynomial-toll-function}]
  To prove Theorem~\ref{theorem:divide-conquer-polynomial-toll-function} using
  Theorem~\ref{thm:asymp}, we need to determine the eigenvalues of $C$ which
  are greater or equal than the joint spectral radius of $A_0$ and $A_1$ (with the
  notations of Lemma~\ref{lemma:h-spectrum-joint-spectral-radius}) and the size
  of the largest Jordan block associated with any such eigenvalue.

  As $\rho(A)=\max\set{\alpha, \beta, 2^{k-1}}\ge 2^{k-1}\ge 1$, it is
  clear that the only relevant eigenvalues of $C$ are contained in the set
  $\set{\alpha+\beta, 2^k, 2^{k-1}}$.

  The main case distinction of
  Theorem~\ref{theorem:divide-conquer-polynomial-toll-function} concerns the
  order of $\alpha+\beta$, $2^k$, and $2^{k-1}$.

  \begin{itemize}
  \item \textbf{Case 1:} $\alpha+\beta>2^k$. This implies that
    $\max\set{\alpha, \beta}>2^{k-1}$ and therefore $\rho(A)
    =\max\set{\alpha, \beta}$ and $A$ has the simple growth property. The eigenvalue $\alpha+\beta$ of $C$ is
    larger than the joint spectral radius and is a simple
    eigenvalue. The eigenvalue $2^k$ is also a simple eigenvalue. If it is
    larger than the joint spectral radius, we are in Case~1a and have two
    asymptotic terms larger than the error term. If $2^k\le \max\set{\alpha,
    \beta}$, we are in Case~1b and have one asymptotic term larger than the
    error term; there is a logarithmic factor in the error term if and only if
    $2^k$ equals the joint spectral radius.
  \item \textbf{Case 2:} $\alpha+\beta=2^k$. In this case,
    $2^k=\alpha+\beta$ has algebraic multiplicity $2$ as an eigenvalue of $C$. We
    note that $C-2^kI$ has the shape
    \begin{equation*}
      C - 2^kI =
      \begin{pmatrix}
        2^k-2^k & \binom{k}{k-1}2^kc_k & b' \\
        0 & 2^k-2^k & b'' \\
        0 & 0 & C'
      \end{pmatrix}=
      \begin{pmatrix}
        0 & k2^kc_k & b' \\
        0 & 0 & b'' \\
        0 & 0 & C'
      \end{pmatrix}
    \end{equation*}
    for some vectors $b'$, $b''$ and some upper triangular matrix $C'$ with
    diagonal elements $2^{k-1}-2^k$, $2^{k-2}-2^k$, \dots, $2^{0}-2^k$. As $c_k\neq 0$ by
    assumption, the kernel of $C-2^kI$ has dimension $1$. We conclude that
    the size
    $m_C(2^k)$ of the largest Jordan block of $C$ associated with the eigenvalue 
    $2^k$ equals $2$. 
    So we have a logarithmic factor in the asymptotic main term.

    We also note that $\max\set{\alpha, \beta}\ge 2^{k-1}$ holds in this case
    with equality if and only if $\alpha=\beta=2^{k-1}$. So
    the joint spectral radius $\rho(A)$ equals $\max\set{\alpha,
    \beta}$ and $A$ has the simple growth property unless $\alpha=\beta=2^{k-1}$.

  \item \textbf{Case 3:} $2^k>\alpha+\beta>2^{k-1}$. In this case, we have
    $\max\set{\alpha, \beta}>2^{k-2}$ and $C$ has two simple dominant eigenvalues
    $2^k$ and $\alpha+\beta$. We do not have additional information about the
    joint spectral radius. If $\max\set{\alpha,\beta}\neq 2^{k-1}$, then $A$ has
    the simple growth property. If $\max\set{\alpha,\beta}< 2^{k-1}$, then
    the joint spectral radius of $A$ is $2^{k-1}$. As $C$ has $2^{k-1}$ as an eigenvalue as well,
    there is a logarithmic factor in
    the error term in exactly this situation.
  \item \textbf{Case 4:} $2^{k-1}\ge \alpha+\beta$. In this case, we have
    $\max\set{\alpha, \beta}<2^{k-1}$, so $\rho(A)=2^{k-1}$ and $A$ has the
    simple growth property. There is
    only one eigenvalue of $C$ larger than this joint spectral radius, namely
    $2^k$. We have to determine $m_C(2^{k-1})$ in order to find out the
    exponent of $\log n$ in the error term. The algebraic multiplicity of
    $2^{k-1}$ as an eigenvalue of $C$ equals
    $1+\iverson{\alpha+\beta=2^{k-1}}$. So if $\alpha+\beta<2^{k-1}$, we
    have $m_C(2^{k-1})=1$ and a factor $\log n$ in the error term.

    We now consider the case $\alpha+\beta=2^{k-1}$ and $k\ge 2$. We note that $C-2^{k-1}I$
    has the shape
    \begin{align*}
      C - 2^{k-1}I &=
      \begin{pmatrix}
        2^{k-1}-2^{k-1} & \binom{k}{k-1}2^{k}c_k & \binom{k-1}{k-2}2^{k-1}c_{k-1} + \binom{k}{k-2}2^{k}c_k  &b' \\
        0 & 2^k-2^{k-1} & \binom{k-1}{k-2}2^{k-2} & b'' \\
        0 & 0 & 2^{k-1}-2^{k-1}& b'''\\
        0 & 0 & 0& C'
      \end{pmatrix}\\
      &=
      \begin{pmatrix}
        0 & k2^{k}c_k & (k-1)2^{k-1}c_{k-1} + k(k-1)2^{k-1}c_k  &b' \\
        0 & 2^{k-1} & (k-1)2^{k-2} & b'' \\
        0 & 0 & 0& b'''\\
        0 & 0 & 0& C'
      \end{pmatrix}
    \end{align*}
    for suitable vectors $b'$, $b''$, $b'''$ and a regular upper triangular
    matrix $C'$. Subtracting $2kc_k$ times the second row from the first row
    does not change the kernel, so we get
    \begin{equation*}
      \ker(C-2^{k-1}I)=\ker 
      \begin{pmatrix}
        0 & 0 & (k-1)2^{k-1}c_{k-1}  &b'-2kc_kb'' \\
        0 & 2^{k-1} & (k-1)2^{k-2} & b'' \\
        0 & 0 & 0& b'''\\
        0 & 0 & 0& C'
      \end{pmatrix}.
    \end{equation*}
    We conclude that $\dim \ker (C-2^{k-1}I)=1+\iverson{c_{k-1}=0}$ and
    therefore $m_C(2^{k-1})=1+\iverson{c_{k-1}\neq 0}$.

    Finally we turn to the case that $\alpha+\beta=2^{k-1}$ and $k=1$. Then we have
    \begin{equation*}
      C-I=
      \begin{pmatrix}
        0 &2c_1&d_0+d_1\\
        0 & 1 & 0 \\
        0 & 0 & 0
      \end{pmatrix}.
    \end{equation*}
    If $d_0=d_1=0$, by the last statement of Lemma~\ref{lemma:h-linear-representation},
    the last row and column of $C-I$ are omitted and $1$
    has algebraic multiplicity $1$ as an eigenvalue of $C$.
    We conclude that $\dim \ker C-I = 1+\iverson{d_0+d_1=0}-\iverson{d_0=0}\iverson{d_1=0}$ and therefore
    $m_C(2^{k-1})=1+\iverson{d_0+d_1\neq 0}$.

    So, to summarise, $m_C(2^{k-1})= 1 + E$ where $E$ is defined in
    Theorem~\ref{theorem:divide-conquer-polynomial-toll-function}.
  \end{itemize}
  
\end{proof}

\begin{proof}[Proof of Theorem~\ref{theorem:constant-toll-function}]
  If $d_0=d_1=0$, then Lemma~\ref{lemma:h-linear-representation} yields
  $A_0=(\beta)$, $A_1=(\alpha)$, and $C=(\alpha+\beta)$, so the joint spectral
  radius of $A$ equals $\max\set{\alpha, \beta}$ which is strictly
  less than the unique eigenvalue $\alpha+\beta$ of $C$. As there is no
  eigenvalue of $C$ less than the joint spectral radius of $A$, there is no
  error term.
  The result follows in
  this case.

  From now on, we assume that $d_0\neq 0$ or $d_1\neq
  0$. Lemma~\ref{lemma:h-linear-representation} yields
  \begin{align*}
    A_0 &= \begin{pmatrix}\beta&d_0\\0&1\end{pmatrix},&
    A_1 &= \begin{pmatrix}\alpha&d_1\\0&0\end{pmatrix},&
    C &= \begin{pmatrix}\alpha+\beta&d_0+d_1\\0&1\end{pmatrix}.
  \end{align*}
  We see that the joint spectral radius of $A$ is
  $\max\set{\alpha,\beta, 1}$ and that $C$ has eigenvalues $\alpha+\beta$ and $1$.
  It is now easy to deduce the assertions of the theorem.
\end{proof}

\section{Details on Example~\ref{example:find-minimum-maximum}}\label{section:details-example}
We proceed as outlined in Remark~\ref{remark:divide-and-conquer-general-remark}.
Setting $x(0)=0$ as usual, we have
\begin{equation*}
  x(n)=x(\floor{n/2})+x(\ceiling{n/2})+2 - \iverson{n=2} -2\iverson{n=1}-2\iverson{n=0}
\end{equation*}
for $n\ge 0$. Equivalently, we have
\begin{align*}
  x(2n)&=2x(n)+2 - \iverson{n=1} - 2\iverson{n=0},\\
  x(2n+1)&=x(n)+x(n+1) +2 -2\iverson{n=0}
\end{align*}
for $n\ge 0$.
Setting $h(n)\coloneqq x(n+1)-x(n)$ leads to
\begin{align*}
  h(2n)&=h(n)+\iverson{n=1},\\
  h(2n+1)&=h(n)+\iverson{n=0}
\end{align*}
for $n\ge 0$.
This defines a $2$-regular sequence with a linear representation $(u, A, w)$
with associated right vector-valued sequence $v$ defined by $v(n)=(h(n),
\delta_1(n), \delta_0(n))$ with
\begin{align*}
  A_0&=
       \begin{pmatrix}
         1&1&0\\
         0&0&0\\
         0&0&1
       \end{pmatrix},
            &
              A_1&=\begin{pmatrix}
                1&0&1\\
                0&0&1\\
                0&0&0
              \end{pmatrix},\\
  u&=(1, 0, 0),&
                 w&=(0, 0, 1)^\top.
\end{align*}
Here, $\delta_1$ is defined by $\delta_1(n)\coloneqq \iverson{n=1}$ for $n\ge 0$.

We can now use SageMath\footnote{The code for this example
  is available
  at~\url{https://arxiv.org/src/2403.06589/anc}; it uses the
  code accompanying~\cite{Heuberger-Krenn:2018:asy-regular-sequences}
  which is available at
  at~\url{https://gitlab.com/dakrenn/regular-sequence-fluctuations}.}
to compute Fourier coefficients and to produce Figure~\ref{fig:minmax_example}.

\bibliography{bib/cheub}
\end{document}